\documentclass[10pt]{article}
\usepackage{amsmath}
\usepackage{amssymb}
\usepackage{amsthm}
\usepackage{mathrsfs}

\begin{document}
\title{The boundedness of Bochner-Riesz operators on the weighted weak Hardy spaces}
\author{Hua Wang \footnote{E-mail address: wanghua@pku.edu.cn.}\\
\footnotesize{Department of Mathematics, Zhejiang University, Hangzhou 310027, China}}
\date{}
\maketitle

\begin{abstract}
Let $w$ be a Muckenhoupt weight and $WH^p_w(\mathbb R^n)$ be the weighted weak Hardy spaces. In this paper, by using the atomic decomposition of $WH^p_w(\mathbb R^n)$, we will show that the maximal Bochner-Riesz operators $T^\delta_*$ are bounded from $WH^p_w(\mathbb R^n)$ to $WL^p_w(\mathbb R^n)$ when $0<p\le1$ and $\delta>n/p-{(n+1)}/2$. Moreover, we will also prove that the Bochner-Riesz operators $T^\delta_R$ are bounded on $WH^p_w(\mathbb R^n)$ for $0<p\le1$ and $\delta>n/p-{(n+1)}/2$. Our results are new even in the unweighted case.
\end{abstract}

\section{Introduction and main results}

The Bochner-Riesz operators of order $\delta>0$ in $\mathbb R^n$ are defined initially for Schwartz functions in terms of Fourier transforms by
\begin{equation*}
\big(\widehat{T^\delta_R f}\big)(\xi)=\Big(1-\frac{|\xi|^2}{R^2}\Big)^\delta_+\hat{f}(\xi), \quad 0<R<\infty,
\end{equation*}
where $\hat{f}$ denotes the Fourier transform of $f$. The associated maximal Bochner-Riesz operator is defined by
\begin{equation*}
T^\delta_*f(x)=\sup_{R>0}\big|T^\delta_Rf(x)\big|.
\end{equation*}

These operators were first introduced by Bochner \cite{bochner} in connection with summation of multiple Fourier series and played an important role in harmonic analysis. The problem concerning the spherical convergence of Fourier integrals have led to the study of their $L^p$ boundedness. As for their $H^p$ boundedness, Sj\"olin \cite{sjolin} and Stein, Taibleson and Weiss \cite{stein} proved the following theorem (see also [9, page 130]).

\newtheorem*{thmi}{Theorem I}
\begin{thmi}
Suppose that $0<p\le1$ and $\delta>n/p-(n+1)/2$. Then there exists a constant $C>0$ independent of $f$ and $R$ such that
\begin{equation*}
\big\|T^\delta_R(f)\big\|_{H^p}\le C\|f\|_{H^p}.
\end{equation*}
\end{thmi}

In \cite{stein}, the authors also considered weak type estimate for the maximal Bochner-Riesz operator $T^\delta_*$ at the critical index $\delta=n/p-(n+1)/2$ and showed the following inequality is sharp.

\newtheorem*{thmii}{Theorem II}
\begin{thmii}
Suppose that $0<p<1$ and $\delta=n/p-(n+1)/2$. Then there exists a constant $C>0$ independent of $f$ such that
\begin{equation*}
\sup_{\lambda>0}\lambda^p\big|\big\{x\in\mathbb R^n:T^\delta_*f(x)>\lambda\big\}\big|\le C\|f\|^p_{H^p}.
\end{equation*}
\end{thmii}

In 1995, Sato \cite{sato1} studied the weighted case and obtained the following weighted weak type estimate for the maximal Bochner-Riesz operator $T^\delta_*$.

\newtheorem*{thmiii}{Theorem III}
\begin{thmiii}
Let $w\in A_1$$($Muckenhoupt weight class$)$, $0<p<1$ and $\delta=n/p-(n+1)/2$. Then there exists a constant $C>0$ independent of $f$ such that
\begin{equation*}
\sup_{\lambda>0}\lambda^p\cdot w\big(\big\{x\in\mathbb R^n:T^\delta_*f(x)>\lambda\big\}\big)\le C\|f\|_{H^p_w}^p.
\end{equation*}
\end{thmiii}

In 2006, Lee \cite{lee} considered values of $\delta$ greater than the critical index $n/p-(n+1)/2$ and proved the following weighted strong type estimate.

\newtheorem*{thmiv}{Theorem IV}
\begin{thmiv}
Let $w\in A_1$, $0<p\le1$ and $\delta>n/p-(n+1)/2$. Then there exists a constant $C>0$ independent of $f$ such that
\begin{equation*}
\big\|T^\delta_*(f)\big\|_{L^p_w}\le C\|f\|_{H^p_w}.
\end{equation*}
\end{thmiv}

Furthermore, by using the above $H^p_w$--$L^p_w$ boundedness of the maximal Bochner-Riesz operator $T^\delta_*$, Lee \cite{lee} also obtained the $H^p_w$ boundedness of the Bochner-Riesz operator.

\newtheorem*{thmv}{Theorem V}
\begin{thmv}
Let $w\in A_1$ with critical index $r_w$ for the reverse H\"older condition, $0<p\le1$, $\delta>\max\{n/p-(n+1)/2,[n/p]r_w/(r_w-1)-(n+1)/2\}$. Then there exists a constant $C>0$ independent of $f$ and $R$ such that
\begin{equation*}
\big\|T^\delta_R(f)\big\|_{H^p_w}\le C\|f\|_{H^p_w}.
\end{equation*}
\end{thmv}

The aim of this paper is to investigate some corresponding estimates of (maximal) Bochner-Riesz operators on the weighted weak Hardy spaces $WH^p_w(\mathbb R^n)$ (see Section 2 for the definition).
Our main results are formulated as follows.

\newtheorem{theorem}{Theorem}[section]

\begin{theorem}
Let $0<p\le1$, $\delta>n/p-(n+1)/2$ and $w\in A_1$. Then there exists a constant $C>0$ independent of $f$ such that
\begin{equation*}
\big\|T^\delta_*(f)\big\|_{WL^p_w}\le C\|f\|_{WH^p_w}.
\end{equation*}
\end{theorem}

\begin{theorem}
Let $0<p\le1$, $\delta>n/p-(n+1)/2$ and $w\in A_1$. Suppose that $\delta-{(n-1)}/2$ is not a positive integer, then there exists a constant $C>0$ independent of $f$ and $R$ such that
\begin{equation*}
\big\|T^\delta_R(f)\big\|_{WH^p_w}\le C\|f\|_{WH^p_w}.
\end{equation*}
\end{theorem}

In particular, if we take $w$ to be a constant function, then we immediately get the following

\newtheorem{cor}[theorem]{Corollary}
\begin{cor}
Let $0<p\le1$ and $\delta>n/p-(n+1)/2$. Then there exists a constant $C>0$ independent of $f$ such that
\begin{equation*}
\big\|T^\delta_*(f)\big\|_{WL^p}\le C\|f\|_{WH^p}.
\end{equation*}
\end{cor}

\begin{cor}
Let $0<p\le1$ and $\delta>n/p-(n+1)/2$. Suppose that $\delta-{(n-1)}/2$ is not a positive integer, then there exists a constant $C>0$ independent of $f$ and $R$ such that
\begin{equation*}
\big\|T^\delta_R(f)\big\|_{WH^p}\le C\|f\|_{WH^p}.
\end{equation*}
\end{cor}

\section{Notations and preliminaries}

The definition of $A_p$ class was first used by Muckenhoupt \cite{muckenhoupt}, Hunt, Muckenhoupt and Wheeden \cite{hunt}, and Coifman and Fefferman \cite{coifman} in the study of weighted
$L^p$ boundedness of Hardy-Littlewood maximal functions and singular integrals. Let $w$ be a nonnegative, locally integrable function defined on $\mathbb R^n$; all cubes are assumed to have their sides parallel to the coordinate axes.
We say that $w\in A_p$, $1<p<\infty$, if
\begin{equation*}
\left(\frac1{|Q|}\int_Q w(x)\,dx\right)\left(\frac1{|Q|}\int_Q w(x)^{-1/{(p-1)}}\,dx\right)^{p-1}\le C \quad\mbox{for every cube}\; Q\subseteq \mathbb
R^n,
\end{equation*}
where $C$ is a positive constant which is independent of the choice of $Q$.

For the case $p=1$, $w\in A_1$, if
\begin{equation*}
\frac1{|Q|}\int_Q w(x)\,dx\le C\cdot\underset{x\in Q}{\mbox{ess\,inf}}\,w(x)\quad\mbox{for every cube}\;Q\subseteq\mathbb R^n.
\end{equation*}
A weight function $w\in A_\infty$ if it satisfies the $A_p$ condition for some $1<p<\infty$. It is well known that if $w\in A_p$ with $1<p<\infty$, then $w\in A_r$ for all $r>p$, and $w\in A_q$ for some $1<q<p$. We thus write $q_w\equiv\inf\{q>1:w\in A_q\}$ to denote the critical index of $w$.

Given a cube $Q$ and $\lambda>0$, $\lambda Q$ denotes the cube with the same center as $Q$ whose side length is $\lambda$ times that of $Q$. $Q=Q(x_0,r)$ denotes the cube centered at $x_0$ with side length $r$. For a weight function $w$ and a measurable set $E$, we denote the Lebesgue measure of $E$ by $|E|$ and set the weighted measure $w(E)=\int_E w(x)\,dx$.

We state the following results that will be used in the sequel.

\newtheorem{lemma}[theorem]{Lemma}
\begin{lemma}[\cite{garcia2}]
Let $w\in A_1$. Then, for any cube $Q$, there exists an absolute constant $C>0$ such that
$$w(2Q)\le C\,w(Q).$$
In general, for any $\lambda>1$, we have
$$w(\lambda Q)\le C\cdot\lambda^{n}w(Q),$$
where $C$ does not depend on $Q$ nor on $\lambda$.
\end{lemma}

\begin{lemma}[\cite{garcia2}]
Let $w\in A_q$ with $q>1$. Then, for all $r>0$, there exists a constant $C>0$ independent of $r$ such that
\begin{equation*}
\int_{|x|\ge r}\frac{w(x)}{|x|^{nq}}\,dx\le C\cdot r^{-nq}w\big(Q(0,2r)\big).
\end{equation*}
\end{lemma}

Given a weight function $w$ on $\mathbb R^n$, for $0<p<\infty$, we denote by $L^p_w(\mathbb R^n)$ the weighted space of all functions satisfying
\begin{equation*}
\|f\|_{L^p_w}=\left(\int_{\mathbb R^n}|f(x)|^pw(x)\,dx\right)^{1/p}<\infty.
\end{equation*}
When $p=\infty$, $L^\infty_w(\mathbb R^n)$ will be taken to mean $L^\infty(\mathbb R^n)$, and
\begin{equation*}
\|f\|_{L^\infty_w}=\|f\|_{L^\infty}=\underset{x\in\mathbb R^n}{\mbox{ess\,sup}}\,|f(x)|.
\end{equation*}
We also denote by $WL^p_w(\mathbb R^n)$ the weighted weak $L^p$ space which is formed by all functions satisfying
\begin{equation*}
\|f\|_{WL^p_w}=\sup_{\lambda>0}\lambda\cdot w\big(\big\{x\in\mathbb R^n:|f(x)|>\lambda\big\}\big)^{1/p}<\infty.
\end{equation*}

Let us now turn to the weighted weak Hardy spaces, which are good substitutes for the weighted
Hardy spaces in the study of the boundedness of some operators. The weak $H^p$ spaces have first appeared in the work of Fefferman, Rivi\`ere and Sagher \cite{cfefferman}. The atomic decomposition theory of weak $H^1$ space on $\mathbb R^n$ was given by Fefferman and Soria \cite{rfefferman}. Later, Liu \cite{liu} established the weak $H^p$ spaces on homogeneous
groups. In 2000, Quek and Yang \cite{quek} introduced the weighted weak Hardy spaces $WH^p_w(\mathbb R^n)$ and established their atomic decompositions. Moreover, by using the atomic decomposition theory of $WH^p_w(\mathbb R^n)$, Quek and Yang \cite{quek} also studied the boundedness of Calder\'on-Zygmund-type operators on these spaces.

We write $\mathscr S(\mathbb R^n)$ to denote the Schwartz space of all rapidly decreasing smooth functions and $\mathscr S'(\mathbb R^n)$ to denote the space of all tempered distributions, i.e., the topological dual of $\mathscr S(\mathbb R^n)$. Let $w\in A_\infty$, $0<p\le1$ and $N\ge[n(q_w/p-1)]$. Define
\begin{equation*}
\mathscr A_{N,w}=\Big\{\varphi\in\mathscr S(\mathbb R^n):\sup_{x\in\mathbb R^n}\sup_{|\alpha|\le N+1}(1+|x|)^{N+n+1}|D^\alpha\varphi(x)|\le1\Big\},
\end{equation*}
where $\alpha=(\alpha_1,\dots,\alpha_n)\in(\mathbb N\cup\{0\})^n$, $|\alpha|=\alpha_1+\dots+\alpha_n$, and
\begin{equation*}
D^\alpha\varphi=\frac{\partial^{|\alpha|}\varphi}{\partial x^{\alpha_1}_1\cdots\partial x^{\alpha_n}_n}.
\end{equation*}
For $f\in\mathscr S'(\mathbb R^n)$, the nontangential grand maximal function of $f$ is defined by
\begin{equation*}
G_w f(x)=\sup_{\varphi\in\mathscr A_{N,w}}\sup_{|y-x|<t}\big|(\varphi_t*f)(y)\big|.
\end{equation*}
Then we say that a tempered distribution $f$ belongs to the weighted weak Hardy space $WH^p_w(\mathbb R^n)$ if $G_w f\in WL^p_w(\mathbb R^n)$ and we define $\|f\|_{WH^p_w}=\|G_w f\|_{WL^p_w}$.
In order to simplify the computations, we shall also use another equivalent definition of $WH^p_w(\mathbb R^n)$.
\begin{equation*}
WH^p_w(\mathbb R^n)=\big\{f\in\mathscr S'(\mathbb R^n):G^+_w f\in WL^p_w(\mathbb R^n)\big\},
\end{equation*}
where $G^+_w f$ is called the radial grand maximal function, which is defined by
\begin{equation*}
G^+_w f(x)=\sup_{\varphi\in\mathscr A_{N,w}}\sup_{t>0}\big|(\varphi_t*f)(x)\big|.
\end{equation*}
Moreover, we set $\|f\|_{WH^p_w}=\|G^+_w f\|_{WL^p_w}$.

\begin{theorem}[\cite{quek}]
Let $0<p\le1$ and $w\in A_\infty$. For every $f\in WH^p_w(\mathbb R^n)$, there exists a sequence of bounded measurable functions $\{f_k\}_{k=-\infty}^\infty$ such that

$(i)$ $f=\sum_{k=-\infty}^\infty f_k$ in the sense of distributions.

$(ii)$ Each $f_k$ can be further decomposed into $f_k=\sum_i b^k_i$, where $\{b^k_i\}$ satisfies

\quad $(a)$ Each $b^k_i$ is supported in a cube $Q^k_i$ with $\sum_{i}w(Q^k_i)\le c2^{-kp}$, and $\sum_i\chi_{Q^k_i}(x)\le c$. Here $\chi_E$ denotes the characteristic function of the set $E$ and $c\sim\big\|f\big\|_{WH^p_w}^p;$

\quad $(b)$ $\|b^k_i\|_{L^\infty}\le C2^k;$

\quad $(c)$ $\int_{\mathbb R^n}b^k_i(x)x^\alpha\,dx=0$ for every multi-index $\alpha$ with $|\alpha|\le[n({q_w}/p-1)]$.

Conversely, if $f\in\mathscr S'(\mathbb R^n)$ has a decomposition satisfying $(i)$ and $(ii)$, then $f\in WH^p_w(\mathbb R^n)$. Moreover, we have $\big\|f\big\|_{WH^p_w}^p\sim c.$

It should be pointed out that the moment condition $c)$ can be replaced by

\quad $(c')$ $\int_{\mathbb R^n}b^k_i(x)x^\alpha\,dx=0$ for every multi-index $\alpha$ with $|\alpha|\le s$, and $s\ge[n({q_w}/p-1)]$. This condition will be used in the proofs of our main results.
\end{theorem}

In particular, for $w$ equals to a constant function, we shall denote $WL^p_w(\mathbb R^n)$ and $WH^p_w(\mathbb R^n)$ simply by $WL^p(\mathbb R^n)$ and $WH^p(\mathbb R^n)$.

Throughout this article $C$ denotes a positive constant, which is independent of the main parameters and not necessarily the same at each occurrence.

\section{Some auxiliary lemmas}

The Bochner-Riesz operators can be expressed as convolution operators
\begin{equation*}
T^\delta_Rf(x)=(\phi_{1/R}*f)(x),
\end{equation*}
where $\phi(x)=[(1-|\cdot|^2)^\delta_+]\mbox{\textasciicircum}(x)$ and $\phi_{1/R}(x)=R^n\phi(Rx)$. It is well known that the kernel $\phi$ can be represented as (see \cite{lu2,stein2})
\begin{equation*}
\phi(x)=\pi^{-\delta}\Gamma(\delta+1)|x|^{-(\frac n2+\delta)}J_{\frac n2+\delta}(2\pi|x|),
\end{equation*}
where $J_\mu(t)$ is the Bessel function
\begin{equation*}
J_\mu(t)=\frac{(\frac{t}{2})^\mu}{\Gamma(\mu+\frac12)\Gamma(\frac12)}\int_{-1}^1e^{its}(1-s^2)^{\mu-\frac12}\,ds.
\end{equation*}
The following kernel estimates of these convolution operators are well known. For its proof, we refer the readers to \cite{sato1}. See also [9, page 121].

\begin{lemma}
Let $0<p_1<1$ and $\delta=n/{p_1}-(n+1)/2$. Then the kernel $\phi$ satisfies the inequality
\begin{equation*}
\sup_{x\in\mathbb R^n}(1+|x|)^{n/{p_1}}\big|D^\alpha\phi(x)\big|\le C\quad\mbox{for all multi-indices}\;\alpha.
\end{equation*}
\end{lemma}

Before proving our main theorems, we need to establish the following two auxiliary lemmas.

\begin{lemma}
Let $0<p_1<1$ and $\delta=n/{p_1}-(n+1)/2$. Then for any given function $b\in L^\infty(\mathbb R^n)$ with support contained in $Q=Q(x_0,r)$, and
\begin{equation*}
\int_{\mathbb R^n}b(x)x^\gamma\,dx=0 \quad\mbox{for every multi-index}\;\;|\gamma|\le N_1=[n(1/{p_1}-1)],
\end{equation*}
we have
\begin{equation*}
T^\delta_*(b)(x)\le C\cdot\|b\|_{L^\infty}\frac{r^{n/{p_1}}}{|x-x_0|^{n/{p_1}}}, \quad \mbox{whenever}\; \;|x-x_0|>\sqrt{n}r.
\end{equation*}
Here and in what follows, we always denote $N_1=[n(1/{p_1}-1)]$.
\end{lemma}

\begin{proof}
Actually, this lemma was essentially contained in \cite{lee}. Here we give its proof for completeness. For any $\varepsilon>0$, we write
\begin{equation*}
(\phi_\varepsilon*b)(x)=\varepsilon^{-n}\int_{Q(x_0,r)}\phi\Big(\frac{x-y}{\varepsilon}\Big)b(y)\,dy.
\end{equation*}
Let us now consider the following two cases.

$(i)$ $0<\varepsilon\le r$. Note that $\delta=n/{p_1}-(n+1)/2$, then by Lemma 3.1, we have
\begin{equation*}
\big|(\phi_\varepsilon*b)(x)\big|\le C\cdot\varepsilon^{n/{p_1}-n}\int_{Q(x_0,r)}\frac{|b(y)|}{|x-y|^{n/{p_1}}}\,dy.
\end{equation*}
By our assumption, when $y\in Q(x_0,r)$, then we can easily get $|x-y|\ge|x-x_0|-|y-x_0|\ge\frac{|x-x_0|}{2}$. Observe that $0<p_1<1$, then $n/{p_1}-n>0$. Hence
\begin{equation}
\big|(\phi_\varepsilon*b)(x)\big|\le C\cdot\|b\|_{L^\infty}\frac{r^{n/{p_1}}}{|x-x_0|^{n/{p_1}}}.
\end{equation}

$(ii)$ $\varepsilon>r$. It is easy to see that the choice of $N_1$ in this lemma implies $\frac{n}{n+N_1+1}<p_1\le\frac{n}{n+N_1}$. Using the vanishing moment condition of $b$, Taylor's theorem and Lemma 3.1, we deduce
\begin{equation*}
\begin{split}
\big|(\phi_\varepsilon*b)(x)\big|&=\varepsilon^{-n}\bigg|\int_{Q(x_0,r)}
\Big[\phi\Big(\frac{x-y}{\varepsilon}\Big)-\sum_{|\gamma|\le N_1}\frac{D^\gamma\phi(\frac{x-x_0}{\varepsilon})}{\gamma!}\Big(\frac{y-x_0}{\varepsilon}\Big)^\gamma\Big] b(y)\,dy\bigg|\\
&\le\varepsilon^{-n}\cdot\Big(\frac{{\sqrt n}r}{2\varepsilon}\Big)^{N_1+1}
\int_{Q(x_0,r)}\sum_{|\gamma|=N_1+1}\Big|\frac{D^\gamma\phi(\frac{x-x_0-\theta (y-x_0)}{\varepsilon})}{\gamma!}\Big|\big|b(y)\big|\,dy\\
&\le C\cdot\frac{r^{N_1+1}}{\varepsilon^{n+N_1+1}}\int_{Q(x_0,r)}\Big|\frac{x-x_0-\theta (y-x_0)}{\varepsilon}\Big|^{-n/{p_1}}\big|b(y)\big|\,dy,
\end{split}
\end{equation*}
where $0<\theta<1$. As in the first case $(i)$, we have $|x-x_0|\ge2|y-x_0|$, which gives $\big|x-x_0-\theta (y-x_0)\big|\ge\frac{|x-x_0|}{2}$. So we have
\begin{equation*}
\big|(\phi_\varepsilon*b)(x)\big|\le C\cdot\|b\|_{L^\infty}\frac{r^{n+N_1+1}}{\varepsilon^{n+N_1+1-n/{p_1}}}\frac{1}{|x-x_0|^{n/{p_1}}}.
\end{equation*}
Notice that $n+N_1+1-n/{p_1}>0$, then for any $\varepsilon>r$, we have $\varepsilon^{n+N_1+1-n/{p_1}}>r^{n+N_1+1-n/{p_1}}$. Therefore
\begin{equation}
\big|(\phi_\varepsilon*b)(x)\big|\le C\cdot\|b\|_{L^\infty}\frac{r^{n/{p_1}}}{|x-x_0|^{n/{p_1}}}.
\end{equation}
Summarizing the inequalities (1) and (2) derived above and then taking the supremum over all $\varepsilon>0$, we obtain the desired estimate.
\end{proof}

Furthermore, by using the estimate in Lemma 3.2, we are able to prove the following result.

\begin{lemma}
Let $0<p_1<1$ and $\delta=n/{p_1}-(n+1)/2$. Assume that $n(1/{p_1}-1)$ is not a positive integer and $b$ satisfies the same conditions as in Lemma $3.2$, then for any $R>0$, we have
\begin{equation*}
G^+_w(T^\delta_R b)(x)\le C\cdot\|b\|_{L^\infty}\frac{r^{n/{p_1}}}{|x-x_0|^{n/{p_1}}}, \quad \mbox{whenever}\; \;|x-x_0|>(2\sqrt{n})r.
\end{equation*}
\end{lemma}

\begin{proof}
We first claim that for every multi-index $\gamma$ with $|\gamma|\le N_1$ and for any $0<R<\infty$, the following identity holds
\begin{equation}
\int_{\mathbb R^n}T^\delta_R b(y)y^{\gamma}\,dy=0.
\end{equation}
In fact, by using the Leibnitz's rule and vanishing moment condition of $b$, we thus have
\begin{equation*}
\begin{split}
\int_{\mathbb R^n}T^\delta_R b(y)y^{\gamma}\,dy&=(T^\delta_R b(y)y^{\gamma}){\mbox{\textasciicircum}}(0)\\
&=C\cdot D^{\gamma}(\widehat{T^\delta_R b})(0)\\
\end{split}
\end{equation*}
\begin{equation*}
\begin{split}
&=C\cdot D^{\gamma}(\widehat{\phi_{1/R}}\cdot\widehat b)(0)\\
&=C\cdot\sum_{|\alpha|+|\beta|=|\gamma|}(D^\alpha\widehat{\phi_{1/R}})(0)\cdot(D^\beta\widehat b)(0)\\
&=0.
\end{split}
\end{equation*}
Then for any $t>0$, by the identity (3) derived above, we will split the expression $\varphi_t*(T^\delta_R b)$ into three parts and estimate each term respectively, where $\varphi\in\mathscr A_{N_1,w}$.
\begin{equation*}
\begin{split}
\big|\varphi_t*(T^\delta_R b)(x)\big|=&\bigg|\int_{\mathbb R^n}t^{-n}\Big[\varphi\Big(\frac{x-y}{t}\Big)-\sum_{|\gamma|\le N_1}\frac{D^\gamma\varphi(\frac{x-x_0}{t})}{\gamma!}\Big(\frac{y-x_0}{t}\Big)^\gamma\Big]T^\delta_R b(y)\,dy\bigg|\\
\le&\,t^{-n}\bigg|\int_{|y-x_0|\le{\sqrt n} r}\Big[\varphi\Big(\frac{x-y}{t}\Big)-\sum_{|\gamma|\le N_1}\frac{D^\gamma\varphi(\frac{x-x_0}{t})}{\gamma!}\Big(\frac{y-x_0}{t}\Big)^\gamma\Big]T^\delta_R b(y)\,dy\bigg|\\
&+t^{-n}\bigg|\int_{{\sqrt n} r<|y-x_0|\le\frac{|x-x_0|}{2}}\Big[\varphi\Big(\frac{x-y}{t}\Big)-\sum_{|\gamma|\le N_1}\frac{D^\gamma\varphi(\frac{x-x_0}{t})}{\gamma!}\Big(\frac{y-x_0}{t}\Big)^\gamma\Big]T^\delta_R b(y)\,dy\bigg|\\
&+t^{-n}\bigg|\int_{|y-x_0|>\frac{|x-x_0|}{2}}\Big[\varphi\Big(\frac{x-y}{t}\Big)-\sum_{|\gamma|\le N_1}\frac{D^\gamma\varphi(\frac{x-x_0}{t})}{\gamma!}\Big(\frac{y-x_0}{t}\Big)^\gamma\Big]T^\delta_R b(y)\,dy\bigg|\\
=&\,K_1+K_2+K_3.
\end{split}
\end{equation*}
For the term $K_1$, by using Taylor's theorem, we obtain
\begin{equation*}
K_1\le t^{-n}\cdot\Big(\frac{{\sqrt n}r}{t}\Big)^{N_1+1}
\int_{|y-x_0|\le\sqrt{n}r}\sum_{|\gamma|=N_1+1}\Big|\frac{D^\gamma\varphi(\frac{x-x_0-\theta' (y-x_0)}{t})}{\gamma!}\Big|\big|T^\delta_*(b)(y)\big|\,dy,
\end{equation*}
where $0<\theta'<1$. Since $0<p_1<1$ and $\delta=n/{p_1}-(n+1)/2$, then $\delta>(n-1)/2$. In this case, it is well known that the following inequality (see \cite{lu2,stein2})
\begin{equation}
T^\delta_* f(x)\le C\cdot M f(x),
\end{equation}
holds for any function $f$, where $M$ denotes the Hardy-Littlewood maximal operator. For any point $x$ with $|x-x_0|>(2\sqrt n)r$ and $|y-x_0|\le\sqrt{n}r$, as before, we have $\big|x-x_0-\theta'(y-x_0)\big|\ge\frac{|x-x_0|}{2}$. Observe that $\varphi\in\mathscr A_{N_1,w}$, then it follows immediately from the inequality (4) that
\begin{equation*}
\begin{split}
K_1&\le C\cdot\frac{r^{N_1+1}}{t^{n+N_1+1}}\int_{|y-x_0|\le\sqrt{n}r}\Big|\frac{x-x_0-\theta' (y-x_0)}{t}\Big|^{-n-N_1-1}\big|T^\delta_*(b)(y)\big|\,dy\\
&\le C\cdot\|M(b)\|_{L^\infty}\frac{r^{n+N_1+1}}{|x-x_0|^{n+N_1+1}}\\
&\le C\cdot\|b\|_{L^\infty}\frac{r^{n/{p_1}}}{|x-x_0|^{n/{p_1}}},
\end{split}
\end{equation*}
where in the last inequality we have used the facts that $n+N_1+1-n/{p_1}>0$ and $\big(\frac{r}{|x-x_0|}\big)^{n+N_1+1-n/{p_1}}\le1$.

On the other hand, for the term $K_2$, we note that $|y-x_0|>\sqrt{n}r$ and $\varphi\in\mathscr A_{N_1,w}$, then it follows from Taylor's theorem and Lemma 3.2 that
\begin{equation*}
\begin{split}
K_2&\le t^{-n}\int_{\sqrt{n}r<|y-x_0|\le\frac{|x-x_0|}{2}}
\sum_{|\gamma|=N_1+1}\Big|\frac{D^\gamma\varphi(\frac{x-x_0-\theta' (y-x_0)}{t})}{\gamma!}\Big|\Big|\frac{y-x_0}{t}\Big|^{N_1+1}\big|T^\delta_*(b)(y)\big|\,dy\\
&\le C\cdot\|b\|_{L^\infty}\, t^{-n}\int_{\sqrt{n}r<|y-x_0|\le\frac{|x-x_0|}{2}}\frac{t^{n+N_1+1}}{|x-x_0|^{n+N_1+1}}\cdot
\frac{|y-x_0|^{N_1+1}}{t^{N_1+1}}
\cdot\frac{r^{n/{p_1}}}{|y-x_0|^{n/{p_1}}}\,dy\\
&\le C\cdot\|b\|_{L^\infty}\frac{r^{n/{p_1}}}{|x-x_0|^{n+N_1+1}}\int_{\sqrt{n}r<|y-x_0|\le\frac{|x-x_0|}{2}}
|y-x_0|^{N_1+1-n/{p_1}}\,dy.
\end{split}
\end{equation*}
If we rewrite the integral in polar coordinates and note the fact that $-1<n+N_1-n/{p_1}\le0$, then we can get
\begin{align}
\int_{\sqrt{n}r<|y-x_0|\le\frac{|x-x_0|}{2}}|y-x_0|^{N_1+1-n/{p_1}}\,dy\notag
&\le C\int_{\sqrt{n}r}^{\frac{|x-x_0|}{2}}\rho^{n+N_1-n/{p_1}}\,d\rho\notag\\
&\le C\cdot|x-x_0|^{-n/{p_1}+n+N_1+1}.
\end{align}
Substituting the above inequality (5) into the term $K_2$, we obtain
\begin{equation*}
K_2\le C\cdot\|b\|_{L^\infty}\frac{r^{n/{p_1}}}{|x-x_0|^{n/{p_1}}}.
\end{equation*}
For the last term $K_3$, observe that $|y-x_0|>\frac{|x-x_0|}{2}>\sqrt{n}r$. Therefore, by using Lemma 3.2 and the fact that $\varphi\in\mathscr A_{N_1,w}$, we deduce
\begin{equation*}
\begin{split}
K_3&\le t^{-n}\int_{|y-x_0|>\frac{|x-x_0|}{2}}\bigg\{\Big|\varphi\Big(\frac{x-y}{t}\Big)\Big|+\sum_{j=0}^{N_1}\sum_{|\gamma|=j}
\Big|\frac{D^\gamma\varphi(\frac{x-x_0}{t})}{\gamma!}\Big|\Big|\frac{y-x_0}{t}\Big|^j\bigg\}
\big|T^\delta_*(b)(y)\big|\,dy\\
&\le C\cdot\|b\|_{L^\infty}\int_{|y-x_0|>\frac{|x-x_0|}{2}}\bigg\{\big|\varphi_t(x-y)\big|+
\sum_{j=0}^{N_1}\frac{|y-x_0|^j}{|x-x_0|^{n+j}}
\bigg\}\cdot
\frac{r^{n/{p_1}}}{|y-x_0|^{n/{p_1}}}\,dy\\
&\le C\cdot\|b\|_{L^\infty}\bigg(\frac{r^{n/{p_1}}}{|x-x_0|^{n/{p_1}}}\|\varphi_t\|_{L^1}+
\sum_{j=0}^{N_1}\frac{r^{n/{p_1}}}{|x-x_0|^{n+j}}\int_{|y-x_0|>\frac{|x-x_0|}{2}}\frac{dy}{|y-x_0|^{n/{p_1}-j}}\bigg).
\end{split}
\end{equation*}
By our assumption (say $n(1/{p_1}-1)$ is not a positive integer), we know that $n/{p_1}-n>N_1=[n(1/{p_1}-1)]$. Thus we have $n/{p_1}-n-j>0$ for any $0\le j\le N_1$. Making use of the polar coordinates again, we obtain
\begin{equation*}
\begin{split}
\int_{|y-x_0|>\frac{|x-x_0|}{2}}\frac{dy}{|y-x_0|^{n/{p_1}-j}}&\le C\cdot
\int_{\frac{|x-x_0|}{2}}^\infty\frac{1}{\rho^{n/{p_1}-n-j+1}}\,d\rho\\
&\le C\cdot\frac{1}{|x-x_0|^{n/{p_1}-n-j}}.
\end{split}
\end{equation*}
Hence
\begin{equation*}
K_3\le C\cdot\|b\|_{L^\infty}\frac{r^{n/{p_1}}}{|x-x_0|^{n/{p_1}}}.
\end{equation*}
Therefore, combining the above estimates for $K_1$, $K_2$ and $K_3$, and then taking the supremum over all $t>0$ and all $\varphi\in\mathscr A_{N_1,w}$, we obtain the desired result.
\end{proof}

\section{Proof of Theorem 1.1}

\begin{proof}
For any given $\lambda>0$, we may choose $k_0\in\mathbb Z$ such that $2^{k_0}\le\lambda<2^{k_0+1}$. For every $f\in WH^p_w(\mathbb R^n)$, then by Theorem 2.3, we can write
\begin{equation*}
f=\sum_{k=-\infty}^\infty f_k=\sum_{k=-\infty}^{k_0} f_k+\sum_{k=k_0+1}^\infty f_k=F_1+F_2,
\end{equation*}
where $F_1=\sum_{k=-\infty}^{k_0} f_k=\sum_{k=-\infty}^{k_0}\sum_i b^k_i$, $F_2=\sum_{k=k_0+1}^\infty f_k=\sum_{k=k_0+1}^\infty\sum_i b^k_i$ and $\{b^k_i\}$ satisfies $(a)$, $(b)$ and $(c')$. Then we have
\begin{equation*}
\begin{split}
&\lambda^p\cdot w\big(\big\{x\in\mathbb R^n:|T^\delta_*f(x)|>\lambda\big\}\big)\\
\le\,&\lambda^p\cdot w\big(\big\{x\in\mathbb R^n:|T^\delta_*F_1(x)|>\lambda/2\big\}\big)+\lambda^p\cdot w\big(\big\{x\in\mathbb R^n:|T^\delta_*F_2(x)|>\lambda/2\big\}\big)\\
=\,&I_1+I_2.
\end{split}
\end{equation*}
We first claim that
\begin{equation}
\|F_1\|_{L^2_w}\le C\cdot\lambda^{1-p/2}\|f\|^{p/2}_{WH^p_w}.
\end{equation}
In fact, since $supp\,b^k_i\subseteq Q^k_i=Q(x^k_i,r^k_i)$ and $\|b^k_i\|_{L^\infty}\le C 2^k$ by Theorem 2.3, then it follows from Minkowski's integral inequality that
\begin{equation*}
\begin{split}
\|F_1\|_{L^2_w}&\le\sum_{k=-\infty}^{k_0}\sum_i\|b^k_i\|_{L^2_w}\\
&\le\sum_{k=-\infty}^{k_0}\sum_i\|b^k_i\|_{L^\infty}w\big(Q^k_i\big)^{1/2}.
\end{split}
\end{equation*}
For each $k\in\mathbb Z$, by using the bounded overlapping property of the cubes $\{Q^k_i\}$ and the fact that $1-p/2>0$, we thus obtain
\begin{equation*}
\begin{split}
\|F_1\|_{L^2_w}&\le C\sum_{k=-\infty}^{k_0}2^k\Big(\sum_i w(Q^k_i)\Big)^{1/2}\\
&\le C\sum_{k=-\infty}^{k_0}2^{k(1-p/2)}\|f\|^{p/2}_{WH^p_w}\\
&\le C\sum_{k=-\infty}^{k_0}2^{(k-k_0)(1-p/2)}\cdot\lambda^{1-p/2}\|f\|^{p/2}_{WH^p_w}\\
&\le C\cdot\lambda^{1-p/2}\|f\|^{p/2}_{WH^p_w}.
\end{split}
\end{equation*}
Since $w\in A_1$, then $w\in A_2$. By the conditions, we know that $\delta>{(n-1)}/2$. Hence, it follows from Chebyshev's inequality, the $L^2_w$ boundedness of $M$ and (4) that
\begin{equation*}
\begin{split}
I_1&\le \lambda^p\cdot\frac{4}{\lambda^2}\big\|T^\delta_*(F_1)\big\|^2_{L^2_w}\\
&\le C\cdot\lambda^{p-2}\big\|F_1\big\|^2_{L^2_w}\\
&\le C\|f\|^{p}_{WH^p_w}.
\end{split}
\end{equation*}

Now we turn our attention to the estimate of $I_2$. If we set
\begin{equation*}
A_{k_0}=\bigcup_{k=k_0+1}^\infty\bigcup_i \widetilde{Q^k_i},
\end{equation*}
where $\widetilde{Q^k_i}=Q(x^k_i,\tau^{{(k-k_0)p}/n}(2\sqrt n)r^k_i)$ and $\tau$ is a fixed positive number such that $1<\tau<2$. So we can further decompose $I_2$ as
\begin{equation*}
\begin{split}
I_2&\le\lambda^p\cdot w\big(\big\{x\in A_{k_0}:|T^\delta_*F_2(x)|>\lambda/2\big\}\big)+
\lambda^p\cdot w\big(\big\{x\in (A_{k_0})^c:|T^\delta_*F_2(x)|>\lambda/2\big\}\big)\\
&=I'_2+I''_2.
\end{split}
\end{equation*}
Since $w\in A_1$, then by Lemma 2.1, we can get
\begin{equation*}
\begin{split}
I'_2&\le\lambda^p\sum_{k=k_0+1}^\infty\sum_iw\big(\widetilde{Q^k_i}\big)\\
&\le C\cdot\lambda^p\sum_{k=k_0+1}^\infty\tau^{(k-k_0)p}\sum_iw(Q^k_i)\\
&\le C\|f\|^{p}_{WH^p_w}\sum_{k=k_0+1}^\infty\Big(\frac{\tau}{2}\Big)^{(k-k_0)p}\\
&\le C\|f\|^{p}_{WH^p_w}.
\end{split}
\end{equation*}
In order to estimate the last term $I''_2$, we observe that $\delta>n/p-(n+1)/2$, then we can find a positive number $0<p_1<p\le1$ such that $\delta=n/{p_1}-(n+1)/2$. Set $N_1=[n({q_w}/{p_1}-1)]=[n(1/{p_1}-1)]\ge[n(1/{p}-1)]$. Then we have
\begin{equation*}
\begin{split}
I''_2&\le 2^p\int_{(A_{k_0})^c}\big|T^\delta_*F_2(x)\big|^pw(x)\,dx\\
&\le 2^p\sum_{k=k_0+1}^\infty\sum_i\int_{\big(\widetilde{Q^k_i}\big)^c}\big|T^\delta_*b^k_i(x)\big|^pw(x)\,dx.
\end{split}
\end{equation*}
When $x\in\big(\widetilde{Q^k_i}\big)^c$, then a direct calculation shows that $|x-x^k_i|\ge\tau^{{(k-k_0)p}/n}\sqrt n r^k_i>\sqrt n r^k_i$. Obviously, by Theorem 2.3, we know that all the functions $b^k_i$ satisfy the conditions in Lemma 3.2. Note also that $w\in A_1$ implies $w\in A_{p/{p_1}}$. Applying Lemma 3.2 and Lemma 2.2, we can deduce
\begin{equation*}
\begin{split}
I''_2&\le C\sum_{k=k_0+1}^\infty\sum_i2^{kp}(r^k_i)^{{np}/{p_1}}
\int_{|x-x^k_i|\ge\tau^{{(k-k_0)p}/n}\sqrt n r^k_i}\frac{w(x)}{|x-x^k_i|^{{np}/{p_1}}}\,dx\\
&= C\sum_{k=k_0+1}^\infty\sum_i2^{kp}(r^k_i)^{{np}/{p_1}}
\int_{|y|\ge\tau^{{(k-k_0)p}/n}\sqrt n r^k_i}\frac{w_1(y)}{|y|^{{np}/{p_1}}}\,dy\\
&\le C\sum_{k=k_0+1}^\infty\sum_i2^{kp}\big(\tau^{{(k-k_0)p}/n}\big)^{{-np}/{p_1}}w_1\big(Q(0,\tau^{{(k-k_0)p}/n}\cdot r^k_i)\big)\\
&= C\sum_{k=k_0+1}^\infty\sum_i2^{kp}\big(\tau^{{(k-k_0)p}/n}\big)^{{-np}/{p_1}}w\big(Q(x^k_i,\tau^{{(k-k_0)p}/n}\cdot r^k_i)\big),
\end{split}
\end{equation*}
where $w_1(x)=w(x+x^k_i)$ is the translation of $w(x)$. It is easy to see that $w_1\in A_1$ whenever $w\in A_1$. Therefore, it follows immediately from Lemma 2.1 that
\begin{equation*}
\begin{split}
I''_2&\le C\sum_{k=k_0+1}^\infty\sum_i2^{kp}\big(\tau^{(k-k_0)p}\big)^{1-p/{p_1}}w(Q^k_i)\\
&\le C\|f\|^{p}_{WH^p_w}\sum_{k=k_0+1}^\infty\big(\tau^{(k-k_0)p}\big)^{1-p/{p_1}}\\
&\le C\|f\|^{p}_{WH^p_w},
\end{split}
\end{equation*}
where the last series is convergent since $1-p/{p_1}<0$. Summarizing the above estimates for $I_1$ and $I_2$ and then taking the supremum over all $\lambda>0$, we complete the proof of Theorem 1.1.
\end{proof}

\section{Proof of Theorem 1.2}

\begin{proof}
We follow the strategy of the proof in Theorem 1.1. For any given $\lambda>0$, we are able to choose $k_0\in\mathbb Z$ such that $2^{k_0}\le\lambda<2^{k_0+1}$. For every $f\in WH^p_w(\mathbb R^n)$, as in Theorem 1.1, we have the decomposition
\begin{equation*}
f=\sum_{k=-\infty}^\infty f_k=\sum_{k=-\infty}^{k_0}\sum_i b^k_i+\sum_{k=k_0+1}^\infty\sum_i b^k_i=F_1+F_2.
\end{equation*}
Write
\begin{equation*}
\begin{split}
&\lambda^p\cdot w\big(\big\{x\in\mathbb R^n:|G^+_w(T^\delta_Rf)(x)|>\lambda\big\}\big)\\
\le\,&\lambda^p\cdot w\big(\big\{x\in\mathbb R^n:|G^+_w(T^\delta_RF_1)(x)|>\lambda/2\big\}\big)+\lambda^p\cdot w\big(\big\{x\in\mathbb R^n:|G^+_w(T^\delta_RF_2)(x)|>\lambda/2\big\}\big)\\
=\,&J_1+J_2.
\end{split}
\end{equation*}
Let us first deal with the term $J_1$. For any function $f$, we can easily prove that
\begin{equation*}
G^+_w f(x)\le C\cdot M f(x).
\end{equation*}
Since $w\in A_1$, then $w\in A_2$. Thus, by using Chebyshev's inequality, the $L^2_w$ boundedness of $M$ and the previous inequality (6), we get
\begin{equation*}
\begin{split}
J_1&\le \lambda^p\cdot\frac{4}{\lambda^2}\big\|G^+_w(T^\delta_R F_1)\big\|^2_{L^2_w}\\
&\le C\cdot\lambda^{p-2}\big\|T^\delta_*(F_1)\big\|^2_{L^2_w}\\
&\le C\cdot\lambda^{p-2}\big\|F_1\big\|^2_{L^2_w}\\
&\le C\|f\|^{p}_{WH^p_w}.
\end{split}
\end{equation*}

To estimate the other term $J_2$, we set
\begin{equation*}
B_{k_0}=\bigcup_{k=k_0+1}^\infty\bigcup_i \widetilde{\widetilde{Q^k_i}},
\end{equation*}
where $\widetilde{\widetilde{Q^k_i}}=Q(x^k_i,\tau^{{(k-k_0)p}/n}(4\sqrt n)r^k_i)$ and $\tau$ is also a fixed real number such that $1<\tau<2$. As before, we will further decompose $J_2$ as
\begin{equation*}
\begin{split}
J_2\le&\,\lambda^p\cdot w\big(\big\{x\in B_{k_0}:|G^+_w(T^\delta_RF_2)(x)|>\lambda/2\big\}\big)\\
&+
\lambda^p\cdot w\big(\big\{x\in (B_{k_0})^c:|G^+_w(T^\delta_RF_2)(x)|>\lambda/2\big\}\big)\\
=&\,J'_2+J''_2.
\end{split}
\end{equation*}
Using the same arguments as that of Theorem 1.1, we can see that
\begin{equation*}
\begin{split}
J'_2&\le\lambda^p\sum_{k=k_0+1}^\infty\sum_iw\big(\widetilde{\widetilde{Q^k_i}}\big)\\
&\le C\cdot\lambda^p\sum_{k=k_0+1}^\infty\tau^{(k-k_0)p}\sum_iw(Q^k_i)\\
&\le C\|f\|^{p}_{WH^p_w}.
\end{split}
\end{equation*}
Again, we may find a positive number $0<p_1<p\le1$ such that $\delta=n/{p_1}-(n+1)/2$. Set $N_1=[n(1/{p_1}-1)]\ge[n(1/{p}-1)]$. Thus, it follows from Chebyshev's inequality that
\begin{equation*}
\begin{split}
J''_2&\le 2^p\int_{(B_{k_0})^c}\big|G^+_w(T^\delta_RF_2)(x)\big|^pw(x)\,dx\\
&\le 2^p\sum_{k=k_0+1}^\infty\sum_i
\int_{\big(\widetilde{\widetilde{Q^k_i}}\big)^c}\big|G^+_w(T^\delta_Rb^k_i)(x)\big|^pw(x)\,dx.
\end{split}
\end{equation*}
For any $x\in\big(\widetilde{\widetilde{Q^k_i}}\big)^c$, then a simple computation leads to that $|x-x^k_i|\ge\tau^{{(k-k_0)p}/n}(2\sqrt n)r^k_i>(2\sqrt n)r^k_i$. In addition, we also note that $\delta-{(n-1)}/2=n(1/{p_1}-1)$ is not a positive integer by our assumption. Therefore, using Lemmas 2.1, 2.2 and 3.3 and following along the same lines as that of Theorem 1.1, we finally obtain
\begin{equation*}
\begin{split}
J''_2&\le C\sum_{k=k_0+1}^\infty\sum_i2^{kp}(r^k_i)^{{np}/{p_1}}
\int_{|x-x^k_i|\ge\tau^{{(k-k_0)p}/n}(2\sqrt n)r^k_i}\frac{w(x)}{|x-x^k_i|^{{np}/{p_1}}}\,dx\\
&\le C\sum_{k=k_0+1}^\infty\sum_i2^{kp}\big(\tau^{{(k-k_0)p}/n}\big)^{{-np}/{p_1}}w\big(Q(x^k_i,\tau^{{(k-k_0)p}/n}\cdot 2r^k_i)\big)\\
&\le C\|f\|^{p}_{WH^p_w}.
\end{split}
\end{equation*}
Combining the above estimates for $J_1$ and $J_2$ and then taking the supremum over all $\lambda>0$, we conclude the proof of Theorem 1.2.
\end{proof}

\end{document}